\documentclass[leqno,12pt]{article} 
\usepackage{amsmath, amssymb}
\usepackage{amsthm} 
\usepackage{amssymb}
\usepackage{mathrsfs}
\usepackage{amssymb, url, color, graphicx, amscd, mathrsfs}
\usepackage[colorlinks=true, bookmarks=true, pdfstartview=FitH, pagebackref=true, linktocpage=true, linkcolor = magenta, citecolor = blue]{hyperref}
\usepackage{graphicx}
\usepackage{times}
\theoremstyle{plain} 
\newtheorem{theorem}{\indent\sc Theorem}[section]
\newtheorem{lemma}[theorem]{\indent\sc Lemma}
\newtheorem{corollary}[theorem]{\indent\sc Corollary}
\newtheorem{proposition}[theorem]{\indent\sc Proposition}

\theoremstyle{definition} 

\newtheorem{remark}[theorem]{\indent\sc Remark}

\numberwithin{equation}{section}
\DeclareMathOperator{\Hess}{Hess}
\DeclareMathOperator{\Ric}{Ric}

\makeatletter

\def\address#1#2{\begingroup
\noindent\parbox[t]{7.8cm}{
\small{\scshape\ignorespaces#1}\par\vskip1ex
\noindent\small{\itshape E-mail address}
\/: #2\par\vskip4ex}\hfill
\endgroup}

\makeatother
\title{{Sharp gradient estimates on weighted manifolds with compact boundary}} 
\author{\textsc{Ha Tuan Dung, Nguyen Thac Dung and Jia-Yong Wu} 
}
\date{} 
\begin{document}

\maketitle

\footnote{ 
2010 \textit{Mathematics Subject Classification}.
Primary 58J05; Secondary 58J35}
\footnote{ 
\textit{Key words and phrases}.
Manifold with compact boundary, Heat equation,
Gradient estimate, Bakry-\'Emery Ricci curvature,
weighted mean curvature, Liouville theorem.}

\begin{abstract}
In this paper, we prove sharp gradient estimates for positive solutions to the
weighted heat equation on smooth metric measure spaces with compact boundary.
As an application, we prove Liouville theorems for ancient solutions satisfying
the Dirichlet boundary condition and some sharp growth restriction near infinity.
Our results can be regarded as a refinement of recent results due to Kunikawa and
Sakurai.
\end{abstract}

\section{Introduction}
\quad\quad In geometric analysis, gradient estimates play an important role in the studying
elliptic and parabolic equations on Riemannian manifolds. In \cite{Yau}, Yau proved
that if the Ricci curvature of manifold $(M^n,g)$ satisfies ${\rm Ric}_M\geq-(n-1)K$
for some constant $K\geq 0$, then any positive harmonic function $u$ on $(M^n,g)$ satisfies
\[
\lim\limits_{B_R(x_0)}\frac{|\nabla u|}{u}\le c_n\left(\frac{1}{R}+\sqrt{K}\right),
\]
where $B_R(x_0)$ is a geodesic ball with radius $R$ and center $x_0\in M^n$. As a
consequence, any positive harmonic function on manifolds with non-negative Ricci
curvature must be constant. Later, in \cite{LY86}, Li and Yau derived gradient
estimates for parabolic equations on Riemannian manifolds with or without boundary.
Moreover, they applied parabolic gradient estimates to study upper and lower bounds
of the heat kernel, eigenvalue estimates and Betti numbers estimates on manifolds.
Motivated by Li-Yau's gradient estimate technique, in \cite{Ham93}, Hamilton
established a new gradient estimate for the heat equation. His result may allow one
to compare the temperature of two different points at the same time on compact manifolds.
In \cite{SZ06}, Souplet and Zhang extended Hamilton's estimate to a localized version.
Their result enables the comparison of temperature distribution instantaneously,
without any lag in time, even for noncompact manifolds.

On the other hand, there exist some interesting works regarding to Riemannian manifolds
with boundary. In \cite{Roger}, Chen generalized Li-Yau gradient estimates to compact
Riemannian manifolds with possible non-convex boundary. As an application, he gave a
lower bound of the first Neumann eigenvalue in terms of geometric invariants with the
boundary. In \cite{Wang97}, Wang developed Li-Yau gradient estimates for the heat
equation under Neumann boundary condition. As an application, he obtained upper and
lower bounds for the heat kernel satisfying Neumann boundary conditions on compact
manifolds with non-convex boundary. In \cite{Hsu}, Hsu proved global Li-Yau gradient
estimates for the conjugate heat equation on compact manifolds with boundary, whose
metric evolving by the Ricci flow. In \cite{Olive},  Oliv\'e proved Li-Yau gradient
estimates for the heat equation, with Neumann boundary conditions, on a compact
Riemannian submanifold with boundary,  satisfying some integral Ricci curvature
assumption. For more related works, we refer to \cite{DKN, DLT, Wu15} and the
references therein.

Besides the above results related to Neumann boundary conditions, Kunikawa and Sakurai
\cite{KS20} recently proved Yau type gradient estimates for harmonic functions with
some Dirichlet boundary condition. More precisely, they proved that
\begin{theorem}\label{keita1}
Let $(M, g)$ be an $n$-dimensional, complete Riemannian manifold with compact boundary.
Assume ${\rm Ric}_M\ge-(n-1)K$ for some constant $K\ge0$ and the mean curvature of
$\partial M$ satisfies $H_{\partial M}\ge-(n-1)\sqrt{K}$.
Let $u:B_R(\partial M)\to(0,\infty)$ be a positive harmonic function with Dirichlet
boundary condition (i.e., it is constant on the boundary). If the derivative $u_\nu$
in the direction of the outward unit normal vector $\nu$ is non-negative over $\partial M$,
then
\[
\sup\limits_{B_{R/2}(\partial M)}\frac{|\nabla u|}{u}\leq c_n\left(\frac{1}{R}+\ \sqrt[]{K}\right).
\]
Here $H_{\partial M}$ denotes the mean curvature of $\partial M$ and the Dirichlet boundary
condition means that $u$ is constant on $\partial M$.
\end{theorem}

Kunikawa and Sakurai \cite{KS20} further proved local Souplet-Zhang gradient
estimates for the heat equations with Dirichlet boundary condition.
\begin{theorem}\label{keita2}
Let $(M, g)$ be an $n$-dimensional complete Riemannian manifold with compact boundary.
Assume ${\rm Ric}_M\ge-(n-1)K$ and $H_{\partial M}\ge0$. Let $0<u<A$ for
some constant $A>0$, be a solution to the heat equation on
$Q_{R,T}(\partial M):=B_R(\partial M)\times[-T,0]$. If $u$ satisfies the Dirichlet boundary
condition, and $u_\nu\ge0$ and $\partial_tu\le0$ over $\partial M\times[-T, 0]$, then
\[
\sup\limits_{Q_{R/2, T/4}(\partial M)}\frac{|\nabla u|}{u}\le c_n\left(\frac{1}{R}+\frac{1}{\sqrt{T}}+\sqrt{K}\right)\left(1+\log\frac{A}{u}\right).
\]
\end{theorem}

Inspired by the Kunikawa-Sakurai work, in this paper we will study gradient
estimates for weighted harmonic functions and weighted heat equation with Dirichlet boundary
condition. In particular, our results improve Theorems \ref{keita1} and \ref{keita2}
on the manifold case.

Before stating our results, we fix some notations. A smooth metric measure space (also called
weighted manifold) is a triple $(M,g,e^{-f}dv)$, where $(M,g)$ is a complete $n$-dimensional
Riemannian manifold, $f$ is a smooth real-valued function on $M$ and $dv$ is the volume form
with respect to the metric $g$. Denote by $\nabla$ and $\Hess$ the gradient and Hessian operators.
The weighted Laplacian is defined by $\Delta_f:=\Delta-\nabla f\cdot\nabla$. On $(M,g,e^{-f}dv)$,
a natural generalization of the Ricci curvature is called $m$-Bakry-\'Emery Ricci curvature, defined by
\[
{\rm Ric}_f^m:={\rm Ric}+\Hess f-\frac{{\nabla f \otimes \nabla f}}{{m-n}},\quad n\le m\le\infty.
\]
When $m=n$, we regard $f$ to be constant and ${\rm Ric}_f^m={\rm Ric}$. When $m=\infty$,
we have $\infty$-Bakry-\'Emery Ricci curvature ${\rm Ric}_f={\rm Ric}_f^{\infty}$. For a complete
manifold $M$ with compact boundary $\partial M$, let $H_f:=H-\nabla f\cdot\nu$ be weighted
mean curvature on $\partial M$, where $\nu$ is the outer unit normal vector to $\partial M$
(see e.g. \cite{XMZ, LiWei}) and the mean curvature $H$ is defined with respect to $\nu$.
But if $\nu$ is the inner unit normal vector to $\partial M$, then $H_f:=H+\nabla f\cdot\nu$.
We refer the reader to \cite{Saku} for further discussions on $H_f$ and weighted manifolds
with boundary.

Let us first state gradient estimates for the $f$-harmonic function with Dirichlet boundary
condition on $(M,g,e^{-f}dv)$, which is a mild generalization of Theorem \ref{keita1}.

\begin{theorem}\label{GE1}
Let $(M,g,e^{-f}dv)$ be an $n$-dimensional smooth metric measure space with compact boundary.
Assume ${\rm Ric}_f^m\ge-(m-1)K$ and $H_f\ge-L$ for some constants $K,L\ge0$. Let
$u: B_R(\partial M)\to (0, \infty)$ be a positive $f$-harmonic function (i.e. $\Delta_fu=0$)
with Dirichlet boundary condition. If the derivative $u_\nu$ in the direction of the outward
unit normal vector $\nu$ is non-negative over $\partial M$, then
\[
\sup\limits_{B_{R/2}(\partial M)}\frac{|\nabla u|}{u}\le
c_m\left(\frac{1}{R}+L+\sqrt{K}\right).
\]
\end{theorem}
We do not know if there exists a gradient estimate under a lower bound of ${\rm Ric}_f$.
Our present proof strongly depends on a refined Kato inequality in Lemma \ref{lixd},
which seems to be not true on weighted manifolds (see \cite{Kato}). When $L=(m-1)\sqrt{K}$
and $f$ is constant, our result returns to Theorem \ref{keita1}.

Second, we prove elliptic gradient estimates for the $f$-heat equation with Dirichlet
boundary condition on $(M,g,e^{-f}dv)$ by improving the argument of \cite{DungDung19},
which seems to be new even for manifolds.
\begin{theorem}\label{GE2}
Let $(M, g, e^{-f}dv)$ be a complete smooth metric measure space with compact boundary.
Assume ${\rm Ric}_{f}\ge-(n-1)K$ for some constant $K\ge0$ and $H_f\ge 0$.
Let $0<u\le A$ for some constant $A>0$ be a solution to the $f$-heat
equation $u_t=\Delta_fu$ on $Q_{R, T}(\partial M):=B_R(\partial M)\times[-T, 0]$.
If $u$ satisfies the Dirichlet boundary condition (i.e. $u(\cdot, t)\mid_{\partial M}$
is constant for each fixed $t\in[-T, 0]$), $u_\nu\geq0$ and $\partial_tu\le0$ over
$\partial M\times[-T, 0]$, then
\[
\sup\limits_{Q_{R/2, T/2}(\partial M)}\frac{|\nabla u|}{u}
\le c_n\left(\frac{\sqrt{D}+1}{R}+\frac{1}{\sqrt{T}}+\sqrt{K}\right)\sqrt{1+\log\frac{A}{u}},
\]
where $D=1+\log A-\log(\inf\limits_{Q_{R,T}(\partial M)}u)$.
\end{theorem}	

As a consequence, we obtain the following Liouville theorems.

\begin{corollary}\label{Liouhar}
Let $(M, g, e^{-f}dv)$ be a complete smooth metric measure space with compact boundary.
Assume ${\rm Ric}^m_f\ge 0$ and $H_f\ge0$.  Let
$u: M\to (0, \infty)$ be a positive $f$-harmonic function 
with Dirichlet boundary condition. If $u_\nu\ge 0$ over $\partial M$, then
$u$ is constant.
\end{corollary}

\begin{corollary}\label{liouville}
Let $(M, g, e^{-f}dv)$ be a complete smooth metric measure space with compact boundary.
Assume ${\rm Ric}_f\ge 0$ and $H_f\geq0$. Let $u$ be an ancient solution to
the $f$-heat equation.
\begin{enumerate}
\item Suppose that $u>0$. If $u_\nu\geq0$ and $\partial_tu\leq0$ on $\partial M\times(-\infty, 0]$ and
$u(x,t)=e^{o(\rho(x)+{|t|})}$ near infinity, then $u$ is constant.
\item If $u_\nu\geq0$ and $\partial_tu\leq0$ on $\partial M\times(-\infty, 0]$ and
$u(x,t)={o(\rho(x)+{|t|})}$ near infinity, then $u$ is constant.
\end{enumerate}
Here, $\rho(x)=\rho_{\partial M}(x)$ denotes the Riemannian distance from the boundary.
\end{corollary}

\begin{remark}
We would like to emphasize that Corollary \ref{liouville} is better than those in Corollary 1.7 in
\cite{KS20}. Because gradient estimates in Theorem \ref{GE2} (also for manifolds)
control the growth at infinity better than those in Theorem \ref{keita2}. Indeed,
our gradient estimate is sharp in both spatial and time directions due to an example
that $M=(-\infty, 0]$, $f=2x$, $u=e^{x-t}$ is an ancient positive solution to the
$f$-heat equation, $u_\nu\ge0$, $\partial_tu\le0$ on $\partial M\times(-\infty, 0]$
and its growth near infinity is $e^{|x|+|t|}$.
\end{remark}

The rest of this paper is organized as follows. In Section \ref{sec2}, we will
introduce some basic facts on smooth metric measure space with boundary. In
particular, we will give the Laplacian comparison with boundary. In Section \ref{sec3},
we will prove gradient estimates for $f$-harmonic functions with Dirichlet boundary
condition, i.e. Theorem \ref{GE1}. In Section \ref{sec3}, we will prove Theorem
\ref{GE2} and Corollaries \ref{Liouhar} and \ref{liouville}.


\section{Preliminaries}\label{sec2}
	\setcounter{equation}{0}
\quad\quad We start to introduce some basic facts. For a smooth metric measure space
$(M^n, g, e^{-f})dv$ with boundary $\partial M$, the distance function from
the boundary is defined by
\[
\rho:=\rho_{\partial M}:=d(\cdot, \partial M).
\]
Note that $\rho$ is smooth outside of the cut locus for the boundary (see \cite{Sak17}).
The Laplacian comparison for the distance function on smooth metric measure
spaces with boundary was first proved in \cite{WZZ}; see also \cite{Saku}.
Here, we give a general statement due to \cite{Sakurai17}.
We assume that $x\in B_{R}(\partial M)$, let $B(x, d(x, \partial M))\subset M$ be the largest geodesic ball centered at $x$ and $z=\partial B(x, d(x, \partial M))\cap \partial M$. Suppose that $\gamma_{z, x}(t)$ is the geodesic curve starting from $z$ and connecting $z$ and $x$.  We see that $\gamma'_{z,x}(0)$ is the unit inner normal vector for $\partial M$ at $z$. Since $\gamma_{z,x}(d(x, \partial M))=x$ and $d(x, \partial M)\leq R$, Lemma 6.1 in \cite{Sakurai17} implies the
Laplacian comparison theorem for $\infty$-Bakry-\'{E}mery Ricci curvature, which will be used in our proof of Theorems \ref{GE1} and \ref{GE2}.

\begin{theorem}[\cite{Sakurai17}]\label{Lapcom}
Let $(M^n, g, e^{-f}dv)$ be a complete smooth metric measure space with compact boundary $\partial M$. Assume that ${\rm Ric}_f\geq-K$ and $H_f\ge-L$ for some constants $K\ge0$ and $L\in\mathbb{R}$. Then
\[
\Delta_f\rho(x)\leq KR+L
\]
for all $x\in B_R(\partial M)$.
\end{theorem}
Next, we recall the so-called Reilly formula (see \cite{Lima, PeterLi, XMZ}).
\begin{proposition}\label{Reilly}
For all $\varphi\in\mathcal{C}^\infty(M)$,
\begin{equation*}
\begin{aligned}
(|\varphi|^2)_\nu=&2\varphi_\nu(\Delta_f\varphi-\Delta_{\partial M, f}(\varphi\mid_{\partial M}-\varphi_\nu H_f))
+2g_{\partial M}(\nabla_{\partial M}\varphi\mid_{\partial M}, \nabla_{\partial M}\varphi_\nu)\\
&+{\rm II}(\nabla_{\partial M}\varphi\mid_{\partial M}, \nabla_{\partial M}\varphi_{\partial M}).
\end{aligned}
\end{equation*}
\end{proposition}

For the $f$-harmonic function, we have a useful inequality, which was
proved in Theorem 2.2 of \cite{Li05}.
\begin{lemma}[see \cite{Li05}]\label{lixd}
Let $(M, g, e^{-f}dv)$ be a smooth metric measure space with ${\rm Ric}_f^m\geq-(m-1)K$. If $u$ is a positive $f$-harmonic function and $m\geq n$, then $\phi:=|\nabla\log u|$ satisfies
\[
\Delta_f\phi\geq-(m-1)K\phi-2\frac{m-2}{m-1}\frac{\langle\nabla\phi, \nabla u\rangle}{u}+\frac{\phi^3}{n-1}.
\]
\end{lemma}

For the $f$-heat equation, we have another useful inequality, which was established in
\cite{DungDung19} for the manifold case.
	\begin{lemma}\label{eq:l1}
Under the same assumption as in Theorem \ref{GE2}, let $w={{\left| \nabla h \right|}^2}$,
where $h=\sqrt{1+\log(A/u)}$.  For any $\left( x,t \right)\in Q_{R, T}(\partial M)$,
\begin{equation}\label{eq:2.3}
\Delta_f w-{w_t}\ge-2(n-1)K w
+2\left(2h-\frac{1}{h} \right)\left\langle \nabla w,\nabla h\right\rangle+2\left(2+\frac{1}{h^2} \right){w^2}.
\end{equation}
\end{lemma}
\begin{proof}
Let $u(x,t)$ be a solution to the $f$-heat equation and $0<u\le A$ for some constant $A$  in
$Q_{R, T}(\partial M)$. Let $B=Ae$ and $h=\sqrt{\log(B/u)}=\sqrt{1+\log(A/u)}\ge 1$.
Then $u=Be^{-h^2}$ and $\log u=\log B-{{h}^{2}}$. By the $f$-heat equation, we directly
compute that
\begin{equation}\label{eq:2.2}
{{h}_{t}}=\Delta_f h+|\nabla h|^2\left( \frac{1}{h}-2h \right).
\end{equation}

On the other hand, by the Bochner formula and $\Ric_f\ge-(n-1)K$, we deduce that
\begin{equation*}
\begin{aligned}
\Delta_fw-w_t&= 2|\nabla^2h|^2+2\Ric_f(\nabla h, \nabla h)+2\langle\nabla\Delta_f h,\nabla h\rangle-w_t\\
			&\ge -2(n-1)Kw+2\langle \nabla \Delta_f h,\nabla h\rangle-w_t.
\end{aligned}
\end{equation*}
		By the equality \eqref{eq:2.2}, we obtain
\begin{equation}\label{eq:2.4}
\begin{aligned}
			\Delta_fw-{{w}_{t}}
			&\ge -2(n-1)Kw+2\left\langle \nabla \left( {{h}_{t}}+{{\left| \nabla h \right|}^{2}}\left( 2h-\frac{1}{h} \right) \right),\nabla h \right\rangle -{{w}_{t}}\\
			&\ge -2(n-1)Kw+2\left\langle \nabla \left( {{h}_{t}} \right),\nabla h \right\rangle +2\left( 2h-\frac{1}{h} \right)\left\langle \nabla \left( {{\left| \nabla h \right|}^{2}} \right),\nabla h \right\rangle\\
			&\quad+2{{\left| \nabla h \right|}^{2}}\left\langle \nabla \left( 2h-\frac{1}{h} \right),\nabla h \right\rangle  -{{w}_{t}}.
\end{aligned}
\end{equation}
Observe that $2\langle\nabla(f_t),\nabla h\rangle =(|\nabla h|^2)_t=w_t$,
and
$$\nabla \left( 2h-\frac{1}{h} \right)=2\nabla h+\frac{\nabla h}{{{h}^{2}}}=\left( 2+\frac{1}{{{h}^{2}}} \right)\nabla h.$$
Hence, \eqref{eq:2.4} implies the desired inequality.
\end{proof}

\section{Gradient estimates for harmonic functions}\label{sec3}
\quad\quad In this section we will follow an argument of \cite{KS20}
to prove Theorem \ref{GE1}.
\begin{proof}[Proof of Theorem \ref{GE1}]
Define $G:=(R^2-\rho^2)\phi$, where $\phi$ is as in Lemma \ref{lixd}. Assume $G$
obtains its maximal at $x_1\in B_R(\partial M)$.

\textbf{Case 1: }If $x_1\in B_R(\partial M)\setminus\partial M$, we
may assume that $x_1\notin {\rm Cut}(\partial M)$ by the Calabi's argument.
At $x_1$, we have
$\nabla G=0$ and $\Delta_fG\le0$. Consequently,
$$\frac{\nabla \rho^2}{R^2-\rho^2}=\frac{\nabla \phi}{\phi}, \quad -\frac{\Delta_f\rho^2}{R^2-\rho^2}+\frac{\Delta_f\phi}{\phi}-\frac{2\left\langle\nabla\rho^2, \nabla\phi\right\rangle}{(R^2-\rho^2)\phi}\leq0.$$
Hence
\begin{equation}\label{kse1}
 \frac{\Delta_f\phi}{\phi}-\frac{\Delta_f\rho^2}{R^2-\rho^2}-2\frac{|\nabla\rho^2|^2}{(R^2-\rho^2)^2}\le 0.
\end{equation}
Since  ${\rm Ric}_f^m\ge-(m-1)K$ (and hence ${\rm Ric}_f\ge-(m-1)K$),
by the Laplacian comparison of Theorem \ref{Lapcom}, we have
\begin{equation}\label{kse2}
\Delta_f\rho^2=2|\nabla\rho|^2+2\rho\Delta_f\rho\le 2+2\rho((m-1)KR+L)\le2+2LR+2(m-1)KR^2,
\end{equation}
where we used $|\nabla\rho|=1$ and $\rho\le R$.
Since $|\nabla\rho^2|=4\rho^2|\nabla\rho|^2=4\rho^2$, Lemma \ref{lixd}
together with \eqref{kse1} and \eqref{kse2} implies
\begin{equation}
\begin{aligned}\label{goodterm}
0&\ge\frac{\Delta_f\phi}{\phi}-\frac{2+2LR+2(m-1)KR^2}{R^2-\rho^2}-\frac{8\rho^2}{(R^2-\rho^2)^2}\\
&\ge-(m-1)K-\frac{2(m-2)}{m-1}\frac{\langle\nabla\phi, \nabla u\rangle}{\phi u}+\frac{\phi^2}{m-1}\\
&\quad-\frac{2+2LR+2(m-1)KR^2}{R^2-\rho^2}-\frac{8\rho^2}{(R^2-\rho^2)^2}.
\end{aligned}
\end{equation}
at $x_1$. By the definition of $G$ and
\[
\frac{\langle\nabla\phi, \nabla u\rangle}{\phi u}=\frac{2\rho\langle\nabla\phi, \nabla u\rangle}{(R^2-\rho^2)u}
\le\frac{2\rho\phi}{R^2-\rho^2},
\]
\eqref{goodterm} can be written as
\begin{equation*}
\begin{aligned}
0&\ge\frac{G^2}{m-1}-\frac{4(m-2)}{m-1}\rho G-[2+2LR+2(m-1)KR^2](R^2-\rho^2)-8\rho^2-(m-1)K(R^2-\rho^2)^2\\
&\ge\frac{G^2}{m-1}-\frac{4(m-2)}{m-1}R G-[2+2LR+2(m-1)KR^2]R^2-8R^2-(m-1)KR^4\\
&\ge\frac{G^2}{2(m-1)}-[(c_m+2LR)]R^2-3(m-1)KR^4
\end{aligned}
\end{equation*}
at $x_1$, where we used the Cauchy-Schwarz inequality. This gives
\[
G(x_1)\le c_m\left(\sqrt{1+LR}\,R+\sqrt{K}R^2\right).
\]
By the definition of $G$, we conclude
\[
\frac{3R^2}{4}\sup\limits_{B_{R/2}(\partial M)}\frac{|\nabla u|}{u}\le
c_m\left(\sqrt{1+LR}\,R+\sqrt{K}R^2\right)
\]
and the desired result follows.

\textbf{Case 2: }If $x_1\in\partial M$, the proof follows by the argument of \cite{KS20}
which is originated in \cite{LiYau2}. We include it for the completeness.
Indeed, at $x_1$, $G_\nu\geq0$, $\phi_\nu=\frac{G_\nu}{R^2}\geq0$ and hence
$(\phi^2)_\nu\geq0$. Since $u$ satisfies the Dirichlet boundary condition and
$u_\nu\geq0$, then $|\nabla u|=u_\nu$. We note that
$\phi=|\nabla u|$ and $\log u$ also satisfies Dirichlet boundary condition. So
Proposition \ref{Reilly} gives
\begin{align*}
0\le(\phi^2)_\nu&=(|\nabla\log u|^2)_\nu=2(\log u)_\nu(\Delta_f\log u-(\log u)_\nu H_f)\\
&=2\frac{u_\nu}{u}\left(-\frac{|\nabla u|^2}{u^2}-\frac{u_\nu}{u}H_f\right)=2\frac{(u_\nu)^2}{u^2}\left(-\frac{u_{\nu}}{u}-H_f\right),
\end{align*}
which implies
$$\phi(x_1)=\frac{u_\nu}{u}\le - H_f\le L.$$
Hence $G(x_1)=R^2\phi(x_1)\le LR^2$. This gives
$$\frac{3R^2}{4}\sup\limits_{B_{R/2}(\partial M)}\frac{|\nabla u|}{u}\le LR^2.$$
The proof is complete.
\end{proof}

\section{Gradient estimates for heat equations}\label{sec4}
To prove Theorem \ref{GE2}, we introduce a smooth cut-off function as in \cite{KS20}.
This cut-off function is originally used in \cite{LY86} (see also \cite{SZ06}).
	\begin{lemma}\label{eq:l2}
	 There exists a smooth cut-off function $\psi=\psi(x,t)$ supported  in
$Q_{R,T}(\partial M)$ such that
		\begin{itemize}
			\item[(i)] $\psi =\psi \left( \rho_{\partial M}(x),t \right)\equiv \psi \left( r,t \right);\psi \left( r,t \right)=1$ in ${{Q}_{\frac{R}{2},\frac{T}{2}}(\partial M)},\quad 0\le \psi \le 1.$
			\item[(ii)] $\psi$ is decreasing as a radial function in the spatial variables, and $\frac{\partial \psi }{\partial r}=0$ in ${{Q}_{{R}/{2}\;,T}(\partial M)}.$
			\item[(iii)] $\left| \frac{\partial \psi }{\partial t} \right|\frac{1}{{{\psi }^{{1}/{2}\;}}}\le \frac{C}{T}$,
\quad $\left| \frac{\partial \psi }{\partial r} \right|\le \frac{{{C}_{\varepsilon }}{{\psi }^{\varepsilon }}}{{{R}}}\quad \text{and}\quad \left| \frac{{{\partial }^{2}}\psi }{{{\partial }^{2}}{{r}^{2}}} \right|\le \frac{{{C}_{\varepsilon }}{{\psi }^{\varepsilon }}}{{{R}^{2}}}, \quad 0<\varepsilon <1.$
		\end{itemize}
	\end{lemma}

Using the cut-off function, we have
	\begin{lemma}\label{mainestimate}
If $\Phi:=\psi(\Delta_f w-{w_t})+w\Delta_f\psi-w{{\psi}_t}-2\frac{|\nabla\psi|^2}{\psi}w$, then
\[
\psi {w^2}\le c\left(\frac{D^2+1}{R^4}+\frac{1}{T^2}+K^2\right)+\frac{\Phi}{4},
\]
where $D:=1+\log A-\log(\underset{Q_{R,T}(\partial M)}{\mathop{\inf}}\,u)$. Here $c$ denotes a constant
depending only on $n$ whose value may change from line to line in the following.
	\end{lemma}
	\begin{proof}
Plugging \eqref{eq:2.3} into $\Phi$,
		\begin{align}
	\Phi&\ge -2(n-1)K\psi w-2\left( 2h-\frac{1}{h} \right)\left\langle \nabla h,\nabla \psi  \right\rangle w\nonumber\\
		&\quad+2\left( 2+\frac{1}{{{h}^{2}}} \right)\psi {{w}^{2}}+w{{\Delta_f }}\psi -w{{\psi }_{t}}-2\frac{{{\left| \nabla \psi  \right|}^{2}}}{\psi }w.\nonumber
		\end{align}
This is equivalent to
		\begin{align}
		4\psi {{w}^{2}}\le& \frac{2{{h}^{2}}}{1+2{{h}^{2}}}2(n-1)K\psi w-\frac{4h\left( 1-2{{h}^{2}} \right)}{1+2{{h}^{2}}}\left\langle \nabla h,\nabla \psi  \right\rangle w\nonumber\\
		&-\frac{2{{h}^{2}}}{1+2{{h}^{2}}}w{{\Delta_f }}\psi +\frac{2{{h}^{2}}}{1+2{{h}^{2}}}w{{\psi }_{t}}+\frac{4{{h}^{2}}}{1+2{{h}^{2}}}\frac{{{\left| \nabla \psi  \right|}^{2}}}{\psi }w+\Phi.\nonumber
		\end{align}
Since $0<\frac{2{{h}^{2}}}{1+2{{h}^{2}}}\le 1$ and $0<\frac{2}{1+2{{h}^{2}}}\le 2$,  we get
\begin{equation}\label{eq:2.5}
	4\psi {w^2}\le 2(n-1)K\psi w-\frac{4h\left( 1-2{h^2} \right)}{1+2{{h}^{2}}}\langle \nabla h,\nabla \psi\rangle w-\frac{2h^2}{1+2h^2}w\Delta_f\psi+w|{{\psi}_t}|+\frac{2{{| \nabla \psi|}^2}}{\psi }w+\Phi.
\end{equation}
In the following we will estimate each term of the right hand side of \eqref{eq:2.5}.
Since $\Ric_f\ge-(n-1)K$ and $H_f\ge 0$, by the Laplacian comparison in Theorem \ref{Lapcom},  		
$\Delta_f\rho\le(n-1)KR$. Using this and Lemma \ref{eq:l2}, we have	
\begin{equation}\label{eq:2.6}
\begin{aligned}
-\frac{2{{h}^{2}}}{1+2{{h}^{2}}}w\Delta_f \psi &=-\frac{2{{h}^{2}}}{1+2{{h}^{2}}}w\left( {{\psi }_r}\Delta_f\rho+{{\psi }_{rr}}|\nabla \rho|^2\right)\\
&\le \frac{2{{h}^{2}}}{1+2{{h}^{2}}}w\left(|{\psi}_r|(n-1)KR+|{\psi}_{rr}|\right)\\
&\le {{\psi }^{1/2\ }}w\frac{\left| {{\psi }_{rr}} \right|}{{{\psi }^{1/2\ }}}+(n-1)KR{{\psi }^{1/2}}w\frac{\left| {{\psi }_r} \right|}{{{\psi }^{1/2}}}\\
&\le \frac{3}{5}\psi {{w}^{2}}+c\left[ {{\left( \frac{\left| {{\psi }_{rr}} \right|}{{{\psi }^{1/2\ }}} \right)}^2}
+\left((n-1)KR\frac{\left| {{\psi }_r} \right|}{{{\psi }^{1/2\ }}} \right)^2\right]\\
&\le \frac{3}{5}\psi {{w}^{2}}+\frac{c}{{{R}^4}}+cK^2.
\end{aligned}
\end{equation}
By the Young inequality, we also have
\begin{equation}\label{eq:2.7}
\begin{aligned}
-\frac{4h\left( 1-2{{h}^{2}} \right)}{1+2{{h}^{2}}}\left\langle \nabla h,\nabla \psi  \right\rangle w
&\le 4h\frac{\left| 1-2{{h}^{2}} \right|}{1+2{{h}^{2}}}\left| \nabla \psi  \right|\left| \nabla h \right|w
\\&\le 4h\left| \nabla \psi  \right|{{w}^{3/2}}=4h\left| \nabla \psi  \right|{{\psi }^{-3/4\ }}{{\left( \psi {{w}^{2}} \right)}^{3/4\ }}\\
&\le \frac{3}{5}\psi {{w}^{2}}+ch^4\frac{{{\left| \nabla \psi  \right|}^{4}}}{{{\psi }^{3}}}\\
&\le \frac{3}{5}\psi {{w}^{2}}+\frac{c{{D}^{2}}}{{{R}^{4}}},
\end{aligned}
\end{equation}
where $D:=\log B-\log( \underset{Q_{R,T}(\partial M)}{\mathop{\inf }}\,u)$.
By the Cauchy-Schwarz inequality, it is not hard to see that the following estimates hold.
We first estimate that
\begin{equation}\label{eq:2.8}
2(n-1)K\psi w\le \frac{3}{5}\psi {{w}^{2}}+{4cK^{2}};
\end{equation}
for $w\left| {{\psi }_{t}} \right|$, we estimate that	
\begin{equation}\label{eq:2.9}
w\left| {{\psi }_{t}} \right|={{\psi }^{{1}/{2}\;}}w\frac{\left| {{\psi }_{t}} \right|}{{{\psi }^{{1}/{2}\;}}}
\le \frac{3}{5}{{\left( {{\psi }^{{1}/{2}\;}}w \right)}^{2}}+c{{\left( \frac{\left| {{\psi }_{t}} \right|}{{{\psi }^{{1}/{2}\;}}} \right)}^{2}}\le \frac{3}{5}\psi {{w}^{2}}+\frac{c}{T^2};
\end{equation}
we also estimate that
\begin{equation}\label{eq:2.10}
\frac{2|\nabla\psi|^2}{\psi}w=2\left( {{| \nabla \psi|}^2}{{\psi }^{-3/2}} \right)\left( {{\psi }^{1/2\;}}w \right)
\le \frac{3}{5}\psi {w^2}+c\frac{{{| \nabla \psi|}^4}}{{{\psi }^3}}\le \frac{3}{5}\psi {w^2}+\frac{c}{{R^4}}.
\end{equation}	
We substitute \eqref{eq:2.6}-\eqref{eq:2.10} into the right hand side of \eqref{eq:2.5}, and get that		
\[
\psi w^2\le c\left( \frac{D^2+1}{R^4}+\frac{1}{T^2}+K^2\right)+\frac{\Phi}{4}.
\]
and the result follows.
\end{proof}
Now we are ready to prove Theorem \ref{GE2} by improving the argument of \cite{KS20}.
\begin{proof}[Proof of Theorem \ref{GE2}]
We may assume $u$ is non-constant in $Q_{R/2, T/2}$. If $u$ is constant,
it is trivial. Hence we let $\psi w$ be a positive function at a point in $Q_{R/2, T/2}(\partial M)$
and achieves its maximal value at  $\left(x_1,t_1\right)$ in $Q_{R/2, T/2}(\partial M)$.
We first claim $x_1\not\in\partial M$. Indeed, by contradiction, assume that $x_1\in\partial M$.
Then at $(x_1, t_1)$, $(\psi w)_\nu\geq0$, consequently $\psi_\nu w+\psi w_\nu=\psi w_\nu\geq0$;
in particular, $w_\nu\geq0$. Since $w=|\nabla\sqrt{\log(B/u)}|^2$, and $h=\sqrt{\log(B/u)}$ also
satisfies the Dirichlet boundary condition, Proposition \ref{Reilly} implies
$0\le w_\nu=(|\nabla{h}|^2)_\nu=2{h}_\nu(\Delta_f h-h_\nu H_f)$.
Since $u$ satisfies the Dirichlet boundary condition, then $|\nabla u|=u_{\nu}$. Hence,
\[
h_\nu=-\frac{u_\nu}{2u\ \sqrt{\log(B/u)}}=-\frac{|\nabla u|}{2u\ \sqrt{\log(B/u)}}=-w^{1/2}.
\]
We now compute
\begin{equation*}
\begin{aligned}
	\Delta_fh&={\rm div}\left(-\frac{\nabla u}{2u\ \sqrt{\log(B/u)}}\right)-\left\langle\nabla f, \nabla h\right\rangle\\
	&=-\frac{\Delta u}{2u\ \sqrt{\log(B/u)}}-\frac{1}{2}\left\langle\nabla u, \nabla\left(\frac{1}{u\ \sqrt[]{\log(B/u)}}\right)\right\rangle+\left\langle\nabla f, \frac{\nabla u}{2u\ \sqrt[]{\log(B/u)}}\right\rangle\\
	&=-\frac{\Delta_f u}{2u\ \sqrt{\log(B/u)}}-\frac{1}{2}\left(-\frac{|\nabla u|^2}{u^2\sqrt{\log(B/u)}}+\frac{1}{2}\frac{|\nabla u|^2}{u^2\left(\log(B/u)\right)^{3/2}}\right)\\
	&=-\frac{u_t}{2uh}+\left(\frac{2h^2-1}{h}\right)w.
\end{aligned}
\end{equation*}
	Plugging these identities into the above inequality, we obtain
\begin{equation*}
\begin{aligned}
	0&\leq -2w^{1/2}\left(-\frac{u_t}{2uh}+\left(\frac{2h^2-1}{h}\right)w+w^{1/2}H_f\right)\\
	&\leq-2w\left(\left(\frac{2h^2-1}{h}\right)w^{1/2}+H_f\right)
\end{aligned}
\end{equation*}
where we used $u_t\leq0$ in the last inequality. This implies
\[
\left(\frac{2h^2-1}{h}\right)w^{1/2}\leq-H_f\leq0.
\]
Since $h\ge1$, $w=0$ at $(x_1, t_1)$. This means $\psi w\equiv0$ on $Q_{R,T}(\partial M)$. This is a contracdiction.

We have shown that $x_1\not\in\partial M$. By the standard argument of Calabi (see \cite{Cal57}), we may assume that $x_1\not\in\partial M\cup{\rm Cut}(\partial M)$. Now at $\left(x_1,t_1\right)$, Lemma \ref{mainestimate} gives
\begin{equation}\label{econclu}
w^2\le c\left(\frac{D^2+1}{R^4}+\frac{1}{T^2}+K^2 \right)+\frac{\Phi}{4}.
\end{equation}
On the other hand, since $(x_1, t_1)$ is a maximal point, we have
$\nabla(\psi w)=0$, $\Delta_f(\psi w)\le 0$ and $(\psi w)_t\ge 0$
at $(x_1, t_1)$. Thus, at $(x_1,t_1)$,
		$$0\ge {{\Delta_f }}\left( \psi w \right)-{{\left( \psi w \right)}_{t}}=\psi \left( {{\Delta_f }}w-{{w}_{t}} \right)+w\left( {{\Delta_f }}\psi -{{\psi }_{t}} \right)+2\left\langle \nabla w,\nabla \psi  \right\rangle,$$
that is, $\Phi(x_1, t_1)\leq0$. Hence, \eqref{econclu} implies
		$$(\psi w)(x, t)\leq (\psi w)(x_1, t_1)\leq c^{1/2}\left( \frac{D+1}{R^2}+\frac{1}{T}+K \right)$$
for all $(x, t)\in Q_{R,T}(\partial M)$. Since $\psi\equiv1$ in $Q_{R/2, T/2}(\partial M)$, by the definition of $w$, we get
$$\frac{|\nabla u|}{u}\leq 2c^{1/4}\left(\frac{\sqrt{D}+1}{R}+\frac{1}{\sqrt{T}}+\sqrt{K}\right)\sqrt{1+\log\frac{A}{u}}.$$
The proof is complete.
\end{proof}
Corollary \ref{Liouhar} follows by letting $R\to\infty$. So we only give a proof of Corollary \ref{liouville}.
\begin{proof}[Proof of Corollary \ref{liouville}]
Now $K=L=0$ and $u_t=\Delta_f u$. Let $v=u+1$ an then $v_t=\Delta_f v$. Moreover, $u$ and $v$ have
the same growth at infinity. Hence, without loss of generality, we may assume $u\geq 1$.
Fix $(x_0, t_0)$. We apply Theorem \ref{GE2} to $Q_{R,R}(\partial M)=B_R(\partial M)\times [t_0-R, t_0]$
and get
\begin{equation*}
\begin{aligned}
\frac{|\nabla u|}{u}(x_0, t_0)
&\le c_m\left( \frac{\sqrt{1+\log A}}{R}+\frac{1}{\sqrt{R}}\right)\sqrt{1+\log \frac{A}{u(x_0, t_0)}}\\
&\le c_m\left( \frac{\sqrt{o(R+|R|)}}{R}+\frac{1}{\sqrt{R}}\right)\sqrt{o(R+|R|)-\log(u(x_0, t_0))}.
\end{aligned}
\end{equation*}
Letting $R\to\infty$, we get $|\nabla u(x_0, t_0)|=0$ and $u$ is constant because $(x_0, t_0)$ is arbitrary.
The proof of (1) is complete.

The proof of (2) is similar as in \cite{SZ06}, we omit the details.
\end{proof}
\begin{remark}
In Theorem 4.2 of \cite{WZZ}, the authors proved that $\rho$ is finite
if some extra conditions on $K$ and $L$ are given. We point out that our gradient
estimates and Liouville theorems work well in this case. In fact, if $\rho$ is finite,
then all boundary balls $B_{\partial M}(R)$ are the same if $R$ is large. The proof
still work without any change. Therefore, we can let $R\to\infty$ even when $\rho$ is finite.
\end{remark}

	\vskip 0.3cm
	\noindent
\section*{Acknowledgment: }The second named author thanks Prof. Keita Kunikawa for useful comments and suggestions on weighted manifolds with $\infty$-Bakry-\'{E}mery curvature. In particular, he thanks Prof. Keita Kunikawa for pointing out Lemma 6.1 in \cite{Sakurai17} that leads to an improvement of this paper.

	\bigskip
\address{{\it Ha Tuan Dung}\\
Faculty of Mathematics  \\
Hanoi Pedagogical University No. 2  \\
Xuan Hoa, Vinh Phuc, Vietnam
}
{hatuandung.hpu2@gmail.com}
\address{ {\it Nguyen Thac Dung}\\
Faculty of Mathematics - Mechanics - Informatics \\
Hanoi University of Science (VNU) \\
Ha N\^{o}i, Vi\^{e}t Nam and \\
Thang Long Institute of Mathematics and Applied Sciences (TIMAS)\\
Thang Long Univeristy\\
Nghiem Xuan Yem, Hoang Mai\\
HaNoi, Vietnam
}
{dungmath@gmail.com}
\address{{\it Jia-Yong Wu}\\
Department of Mathematics,\\
 Shanghai University, China
}
{wujiayong@shu.edu.cn}
\end{document}